\documentclass{article}
\usepackage{graphicx} 

\usepackage{amsthm}
\usepackage{amsmath}
\usepackage{amssymb}
\usepackage{thmtools}
\usepackage{thm-restate}
\usepackage[backend=biber,
style=alphabetic,
]{biblatex}
\DeclareFieldFormat{postnote}{#1}
\DeclareFieldFormat{multipostnote}{#1}
\usepackage{tikz}
\usepackage{circuitikz}
\usepackage{subfiles}
\usepackage{caption}
\DeclareCaptionFormat{custom}
{%
    \textbf{#1#2}\textit{\small #3}
}
\DeclareFieldFormat{postnote}{#1}

\theoremstyle{definition}
\newtheorem{counter}{bad bad bad}[section]
\newtheorem{lemma}[counter]{Lemma}
\newtheorem{definition}[counter]{Definition}

\newtheorem{remark}[counter]{Remark}

\newtheorem{theorem}[counter]{Theorem}
\newtheorem{fact}[counter]{Fact}
\newtheorem{corollary}[counter]{Corollary}
\newtheorem{hypothesis}[counter]{Hypothesis}
\newtheorem{innercustomthm}{Corollary}
\newenvironment{customcor}[1]
  {\renewcommand\theinnercustomthm{#1}\innercustomthm}
  {\endinnercustomthm}

\newcommand{\K}{\mathcal{K}}
\newcommand{\lang}{\operatorname{L}}
\newcommand{\lk}{\preccurlyeq}

\newcommand{\ba}{\mathbf{a}}
\newcommand{\bb}{\mathbf{b}}
\newcommand{\bx}{\mathbf{x}}
\newcommand{\gS}{\operatorname{S}}

\newcommand{\cof}{\text{cf}}

\newcommand{\mon}{\mathfrak{C}}

\newcommand{\tp}{\operatorname{tp}}

\newbox\noforkbox \newdimen\forklinewidth
\setbox0\hbox{$\textstyle\smile$}
\setbox1\hbox to \wd0{\hfil\vrule width \forklinewidth depth-2pt
 height 10pt \hfil}
\setbox\noforkbox\hbox{\lower 2pt\box1\lower 2pt\box0\relax}
\def\unionstick{\mathop{\copy\noforkbox}\limits}

\def\nonfork_#1{\unionstick_{\textstyle #1}}

\setbox0\hbox{$\textstyle\smile$}
\setbox1\hbox to \wd0{\hfil{\sl /\/}\hfil}
\setbox2\hbox to \wd0{\hfil\vrule height 10pt depth -2pt width
              \forklinewidth\hfil}
\newbox\doesforkbox
\setbox\doesforkbox\hbox{\lower 2pt\box1 \lower 2pt\box2\lower2pt\box0\relax}
\def\nunionstick{\mathop{\copy\doesforkbox}\limits}

\def\fork_#1{\nunionstick_{\textstyle #1}}

\newcommand{\dnf}{\unionstick}

\title{The spectrum of limit models in a first order setting}
\author{Jeremy Beard}
\date{}

\begin{document}

\maketitle

\begin{abstract}
    Originally introduced by Kolmann and Shelah as a surrogate for saturated models \cite{kosh}, limit models have been established as natural and useful objects when studying abstract elementary classes. In \cite{shelah2006first} Shelah began the study of when (multiple notions of) limit models exist for first order theories. In this paper we look at their structure.
    
    In superstable theories it is known that all limit models are isomorphic \cite{van06}, but in the strictly stable case the number of non-isomorphic limit models was not well understood. Here we characterise the full spectrum of limit models in the first order stable setting, by a short and simple argument using only the familiar machinery of stable first order theories:

    \begin{customcor}{\ref{fo-main-result}}\label{maintheorem}
        Let $T$ be a complete $\lambda$-stable theory where $\lambda \geq |\lang(T)| + \aleph_0$. Let $\delta_1, \delta_2 < \lambda^+$ be limit ordinals where $\cof(\delta_1)< \cof(\delta_2)$. Let $N_l$ be a $(\lambda, \delta_l)$-limit model for $l = 1, 2$. Then $N_1$ and $N_2$ are isomorphic if and only if $\cof(\delta_1) \geq \kappa_r(T)$. 

        Moreover, if $\kappa_r(T) = \aleph_\alpha$, there are exactly $|\alpha| + 1$ limit models up to isomorphism.
    \end{customcor}

    In the context of first order stable theories, this reduces the proof of the main result of \cite{bema} from 19 pages to 2. We hope this will make limit models a more comprehensible and accessible tool in first order model theory.
\end{abstract}

\section{Introduction}

In the study of abstract elementary classes (AECs), limit models have proved useful in solving approximations of Shelah's categoricity conjecture, which proposes a generalisation of Morley's categoricity theorem holds in all AECs \cite[Question 6.14(3)]{sh702}. More specifically, proving that (high cofinality) limit models are isomorphic allows one to develop independence relations with similar properties to first order forking \cite[2.8, 2.13]{vas19}. Although limit models were introduced as an analogue for partially saturated models in AECs (see Fact \ref{long-limits-are-saturated}), they are still interesting in a first order setting; Shelah began the study of several different notions of limit model in a first order context in \cite{shelah2006first}, and an example of non-isomorphic limit models was first found in the first order context \cite[6.1]{gvv}.

Given two models $M \preccurlyeq N$ of the same size $\lambda$, $N$ is a \emph{universal extension} of $M$ if it contains a copy of every other size $\lambda$ elementary extension of $M$ via an embedding fixing $M$ (see Definition \ref{limit-def}(1)). A $(\lambda, \delta)$-limit model is the union of an increasing $\delta$-length chain of models universally extending one another, all of size $\lambda$ (see Definition \ref{limit-def}(2)). As mentioned above, finding contexts in which two $\lambda$-limit models are isomorphic is important and thoroughly studied in AECs \cite{shvi99}, \cite[II.11.2]{van02}, \cite[Theorem 1]{van16b}, \cite{bovan}, \cite[2.7]{vas19}. 

The goal of this paper is to present, along with the relevant background on limit models, a short and simple argument characterising the spectrum of limit models of a first order stable theory, for a reader familiar with first order model theory but not necessarily AECs.

\begin{customcor}{\ref{fo-main-result}}\label{maintheorem}
    Let $T$ be a complete $\lambda$-stable theory where $\lambda \geq |\lang(T)| + \aleph_0$. Let $\delta_1, \delta_2 < \lambda^+$ be limit ordinals where $\cof(\delta_1)< \cof(\delta_2)$. Let $N_l$ be a $(\lambda, \delta_l)$-limit model for $l = 1, 2$. Then $N_1$ and $N_2$ are isomorphic if and only if $\cof(\delta_1) \geq \kappa_r(T)$.
\end{customcor}

That is, `long limits are isomorphic, and short limits are non-isomorphic'. Since whenever $\cof(\delta_1)=\cof(\delta_2)$ the limits are isomorphic (see Fact \ref{cfiso}), this gives us a complete understanding of the spectrum of limit models.

With this goal in mind, we present all definitions and results in first order terminology and avoid the language of AECs. For those with interests outside of first order logic, a broader introduction to limit models, along with a more detailed history and a related family of results, can be found in \cite{bema}, but these are entirely unnecessary to understand this paper.

In this first order context, the `long limits are isomorphic' argument (Lemma \ref{fo-long-lims-are-iso}) is straightforward and requires us only to show that when $\cof(\delta) \geq \kappa_r(T)$, the $(\lambda, \delta)$-limit models are saturated. The `short limits are non-isomorphic' argument (Theorem \ref{fo-short-models-are-distinct}) develops the method from \cite[6.1]{gvv}, but rather than showing only that the $(\lambda, \omega)$-limit model is not isomorphic to any of the $(\lambda, \geq \kappa_r(T))$-limit models, we show in fact that all $(\lambda, \mu)$-limit models are non-isomorphic for regular $\mu \leq \kappa_r(T)$. This involves building a chain of models around a \emph{uniform $\mu$-tree} (see Definition \ref{uniform-trees}), which captures a type over a domain of size $\mu$ that the final model, a $(\lambda, \mu)$-limit model, cannot realise, witnessing a failure of $\mu^+$-saturation. This is enough, as the $(\lambda, \theta)$-limit model will be $\theta$-saturated for regular $\theta < \lambda^+$ (see Fact \ref{long-limits-are-saturated}).

Section 2 provides the necessary background. We include sketches of proofs of some elementary results, in the hope that these will be useful for readers that would rather avoid learning the background of AECs necessary to understand the proofs we cite. Section 3 presents the main arguments we described above.

This paper was written in parallel with \cite{bema}, which is also concerned with the spectrum of limit models. While the main result \cite[5.1]{bema} from that paper is more general than Corollary \ref{maintheorem}, the proofs are long and technical - showing `long limits are isomorphic' requires the development of a deep theory of \emph{towers} of models, and showing `short limits are non-isomorphic' involves proving a canonicity result for $\lambda$-non-splitting. To then apply the general result to the first order stable context in \cite[6.22(4)]{bema}, additional work is needed to show that the `cutoff' value ($\kappa(\dnf_f, \operatorname{Mod(T)}_\lambda, \preccurlyeq^u)$ in the notation of \cite{bema}) is $\kappa_r(T)$ \cite[5.22, 6.11]{bema}. Using the familiar machinery of first order stable theories described above, we reduce the technical 19 page proof of \cite{bema} down to just 2 pages. This continues the study of limit models in the first order setting initiated by Shelah \cite{shelah2006first}, and we hope this paper will make the topic more accessible to model theorists with little to no background in AECs.

This paper was written while the author was working on a Ph.D. under the direction of Rami Grossberg at Carnegie Mellon University and I would like to thank Professor Grossberg for his guidance and assistance in my research in general and in this work in particular. I thank Marcos Mazari-Armida for his guidance and advice while this work was an appendix of \cite{bema}. I also thank John Baldwin, Samson Leung, and Wentao Yang for helpful comments regarding the presentation of this paper.

\section{Preliminaries}

We assume some background knowledge of first order model theory, in particular some familiarity with the notions of stable theories, $\mu$-saturated models, and first order non-forking.

We typically use $T$ to denote a first order theory, $\lambda, \mu, \theta$ for cardinals, $\alpha, \beta, \delta$ for ordinals, $\eta, \nu$ for elements of trees. Given sequence $\eta$, $\eta^\wedge i$ is the sequence formed by appending $i$ to the end of $\eta$. We use $\mon$ for the monster model of a complete theory. Given a theory $T$, $\lang(T)$ denotes the language of $T$, and we use the shorthand $|T| = |\lang(T)| + \aleph_0$. We use $M, N$ for models. The model $M$ has universe $|M|$ and cardinality $\|M\|$. We use $M \subseteq N$ if $M$ is a submodel of $N$, and $M \preccurlyeq N$ if $M$ is an elementary substructure of $N$. $a, b, c, d$ will denote singleton elements of models, and $\ba, \bb$ will be tuples of elements (possibly singletons). Similarly use $x, y$ and $\bx$ for variables. We use $\varphi$ for formulas. We use standard abuses of notation such as $a \in M$ or $\ba \in M$. The type of $\ba$ over a set $A$ in $N$ is denoted $\tp(\ba/A, N)$; when working in a monster model $\mon$, $\tp(\ba/A) = \tp(\ba/A, \mon)$. The set of $n$-types over a set $A$ is $\gS^n(A)$, and $\gS(A) = \bigcup_{n<\omega} \gS^n(A)$. A model $M$ is $\mu$-saturated if it realises all 1-types (or equivelently all $(<\omega)$-types) over subsets of size $<\mu$; $M$ is saturated if it is $\|M\|$-saturated. A theory $T$ is $\lambda$-stable if the number of 1-types (or equivelently $(<\omega)$-types) over a set of size $\leq \lambda$ is at most $\lambda$.

This section will be a review of the known theory of limit models. Limit models have traditionally been studied in the broader context of AECs. Since this paper is aimed at first order models theorists, we will only define the necessary notions in the restricted first order setting. The `new' content of this paper is entirely in Section \ref{maintheoremsection}.

For the rest of this section, we assume the following:

\begin{hypothesis}
    $T$ is a complete first order theory with infinite models, and $\lambda$ is a fixed infinite cardinal where $\lambda \geq \|T\|$.
\end{hypothesis} 

\begin{definition}[{\cite{kosh}}]\label{limit-def}

    \begin{enumerate}
        \item Suppose $M, N \models T$ and $\|M\| = \|N\|$. We say $N$ is \emph{universal over $M$} provided $M \lk N$ and for all $M' \in \K$ where $M \lk M'$ and $\|M'\|  = \|N\|$, there exists an elementary embedding $f : M' \rightarrow N$ that fixes $M$. We write $M \lk^u N$ as a shorthand.
    
        \item Suppose $\delta$ is a limit ordinal where $\delta < \lambda^+$. Given $M, N \models T$ with $\|M\| = \|N\| = \lambda$, we say $N$ is a \emph{$(\lambda, \delta)$-limit model over $M$} provided there exists an increasing continuous sequence of models $\langle M_\alpha : \alpha \leq \delta\rangle$ such that $M_0 = M$, $M_\delta = N$ and $M_{\alpha + 1}$ is universal over $M_\alpha$ for all $\alpha < \delta$. We may call $\langle M_\alpha : \alpha \leq \delta\rangle$ a \emph{witness} of this property.

        \item A \emph{$(\lambda, \delta)$-limit model} is a $(\lambda, \delta)$-limit model over some $M \in K_\lambda$.
    
        \item A \emph{$\lambda$-limit model} is a $(\lambda, \delta)$-limit model for some limit $\delta < \lambda^+$. When $\lambda$ is clear from context, we may omit it.
    \end{enumerate}
    
\end{definition}

\begin{remark}
    If $N$ is a limit model over $M$, then $N$ is universal over $M$.
\end{remark}

\begin{remark}\label{alternative-universal}
    The definition of when $N$ is universal over $M$ is equivalent to: for all $f : M \rightarrow N'$, there exists $g : N' \rightarrow N$ with $f^{-1} \subseteq g$ (taking the inverse on the image). This is clearly sufficient (take $f = id_M$). And this is necessary, as given such $f$, you can extend $f$ to an isomorphism $\bar{f} : M' \cong N'$, apply that $N$ is universal over $M$ to get a map $g^*:M' \rightarrow N$, and then take $g = g^* \circ \bar{f}^{-1} : N' \rightarrow N$, which extends $f^{-1}$ as $g^*$ fixes $M$.
\end{remark}

\begin{remark}\label{lims-exist-implies-stable}
    Suppose that for all $M \models T$ with $\|M\| = \lambda$, there exists $N \models M$ universal over $M$ (with $\|N\| = \|M\|$). Then $T$ is stable in $\lambda$, since each such $N$ realises all (1-ary) types over $M$.
\end{remark}

The converse to Remark \ref{lims-exist-implies-stable} is also true. We sketch a proof since it is typically presented in the language of AECs.

\begin{fact}[{\cite[2.9]{grva06a}}]\label{limits-exist}
    Suppose $T$ is stable in $\lambda$. For any $M \models T$ where $\|M\| = \lambda$, there exists $N \models T$ universal over $M$.
    
    Moreover, for any $M \models T$ where $\|M\| = \lambda$ and limit $\delta < \lambda^+$, there exists $N \models T$ such that $N$ is a $(\lambda, \delta)$-limit over $M$.
\end{fact}

\begin{proof}
    Work in a monster model $\mon$. Using stability, we can take a continuous increasing sequence of elementary extensions $\langle M_i : i < \lambda \rangle$ with $M_0 = M$ and where $M_{i+1}$ realises all types over $M_i$. Let $N = \bigcup_{i<\lambda} M_i$. We claim $N$ is universal over $M$.

    Suppose $N' \models T$ with $M \preccurlyeq N'$ and $\|N'\| = \lambda$. Enumerate $N = \{a_i : i < \lambda\}$. By induction define $(N', N)$-mappings $f_i : M \cup \{a_r : r < i \} \rightarrow M_i$. To do this, set $f_0 = id_M$. At limit $i<\lambda$, set $f_i = \bigcup_{r<i} f_r$. For successors, since $M_{i+1}$ realises $f_i(\tp(a_i/M \cup \{a_r : i < r\})$, let $f_{i+1}(a_i)$ be any realisation of this type. This completes the construction, and $\bigcup_{i < \lambda} : N' \rightarrow N$. Therefore $N$ is universal over $M$.

    If $\delta < \lambda^+$ is a limit model, use the first part to define a $\preccurlyeq^u$-increasing continuous chain $\langle M_i : i < \delta \rangle$ with $M_0 = M$. Then $\bigcup_{i < \delta} M_i$ is a $(\lambda, \delta)$-limit model over $M$.
\end{proof}

The following fact tells us that to understand the spectrum of limit models, we only need to consider $(\lambda, \delta)$-limit models when $\delta$ is regular, and that for each regular cardinal the model is unique. The proof is a straightforward back-and-forth argument along a cofinal sequence of the witnessing sequences of models - we sketch it for the convenience of those unfamiliar with limit models.

\begin{fact}[{\cite[Fact 1.3.6]{shvi99}}]\label{cfiso}
    Suppose $T$ is complete. Suppose $\delta_1, \delta_2 < \lambda^+$ are limit ordinals and $\cof(\delta_1) = \cof(\delta_2)$. If $M$ is a $(\lambda, \delta_1)$-limit and $N$ is a $(\lambda, \delta_2)$-limit, then $M \cong N$.
\end{fact}

\begin{proof}
    Let $\langle M_i : i \leq \delta_1 \rangle$, $\langle N_i : i \leq \delta_2 \rangle$ be $\preccurlyeq^u$-increasing continuous sequences with $M_{\delta_1} = M$, $N_{\delta_2} = N$. Taking cofinal sequences, without loss of generality $\delta_1 = \delta_2 = \delta$ is regular. There exists some $M^*$ with $M_0 \preccurlyeq M^*$ and $f^* : N_0 \rightarrow M^*$ an elementary embedding. As $M_0 \preccurlyeq^u M_1$, there exists $\bar{f} : M^* \rightarrow M^1$ fixing $M_0$. Let $f_0 = \bar{f} \circ f^* : N_0 \rightarrow M_1$. Now by universality of $N_1$ over $N_0$, by Remark \ref{alternative-universal} there exists $g_1 : M_1 \rightarrow N_1$ where $f_0^{-1} \subseteq g_1$ (taking the inverse on the image). Again by Remark \ref{alternative-universal}, you can find $f_1:N_1 \rightarrow M_2$ with $(g_1)^{-1} \subseteq f_1$. Continue this process: by induction define continuous and $\subseteq$-increasing sequences of elementary embeddings $f_i : N_i \rightarrow M_{i+1}$ and $g_i : M_i \rightarrow N_i$ for $i \in [1, \delta)$ where $f_r^{-1} \subseteq g_i$ for $r < i$ and $g_i^{-1} \subseteq f_i$ for all $i < \delta$. Do this by taking $g_i = \bigcup_{r<i} g_r$ at limit $i < \delta$, and using Remark \ref{alternative-universal} at all other stages as before. At the end, we have that $f = \bigcup_{i<\delta} f_i : N \rightarrow M$ and $g = \bigcup_{i<\delta} g_i : M \rightarrow N$ are inverses of one another, so $g : M \cong N$ as desired.

    
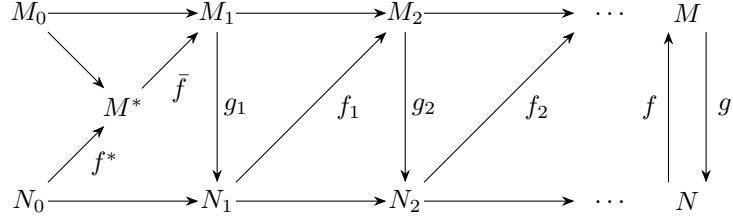
\begin{figure}[!ht]
\centering

\begin{circuitikz}
\tikzstyle{every node}=[font=\normalsize]
\node [font=\normalsize] at (3.75,7.5) {$N_0$};
\node [font=\normalsize] at (3.75,10) {$M_0$};
\node [font=\normalsize] at (5,8.75) {$M^*$};
\draw [->, >=Stealth] (4,9.75) -- (4.75,9);
\draw [->, >=Stealth] (4,7.75) -- (4.75,8.5);
\draw [->, >=Stealth] (4,10) -- (6,10);
\draw [->, >=Stealth] (4,7.5) -- (6,7.5);
\draw [->, >=Stealth] (5.25,9) -- (6,9.75);
\node [font=\normalsize] at (6.25,10) {$M_1$};
\draw [->, >=Stealth] (6.25,9.75) -- (6.25,7.75);
\node [font=\normalsize] at (6.25,7.5) {$N_1$};
\draw [->, >=Stealth] (6.5,7.5) -- (8.5,7.5);
\draw [->, >=Stealth] (6.5,10) -- (8.5,10);
\node [font=\normalsize] at (8.75,10) {$M_2$};
\draw [->, >=Stealth] (9,10) -- (11,10);
\node [font=\normalsize] at (8.75,7.5) {$N_2$};
\draw [->, >=Stealth] (8.75,9.75) -- (8.75,7.75);
\draw [->, >=Stealth] (9,7.5) -- (11,7.5);
\draw [->, >=Stealth] (6.5,7.75) -- (8.5,9.75);
\draw [->, >=Stealth] (9,7.75) -- (11,9.75);
\node [font=\normalsize] at (11.5,10) {$\dots$};
\node [font=\normalsize] at (12.5,10) {$M$};
\draw [->, >=Stealth] (12.25,7.75) -- (12.25,9.75);
\node [font=\normalsize] at (12.5,7.5) {$N$};
\draw [->, >=Stealth] (12.75,9.75) -- (12.75,7.75);
\node [font=\normalsize] at (4.75,8) {$f^*$};
\node [font=\normalsize] at (5.75,9) {$\bar{f}$};
\node [font=\normalsize] at (6.5,8.75) {$g_1$};
\node [font=\normalsize] at (8,8.75) {$f_1$};
\node [font=\normalsize] at (9,8.75) {$g_2$};
\node [font=\normalsize] at (10.5,8.75) {$f_2$};
\node [font=\normalsize] at (11.5,7.5) {$\dots$};
\node [font=\normalsize] at (12,8.75) {$f$};
\node [font=\normalsize] at (13,8.75) {$g$};
\end{circuitikz}

\label{back-and-forth-diagram}
\end{figure}
\end{proof}

The following illustrates part of the relationship between limit models and saturated models.

\begin{fact}\label{long-limits-are-saturated}
    Suppose $M \models T$ and $\mu < \lambda^+$ is an infinite regular cardinal. If $M$ is a $(\lambda, \mu)$-limit model, then $M$ is $\mu$-saturated.
\end{fact}

\begin{proof}
    Let $\langle M_i : i \leq \mu \rangle$ witness that $M$ is a $(\lambda, \mu)$-limit model. Suppose $A \subseteq |M|$ and $|A| < \mu$ and $p \in \gS(A)$. Then $A \subseteq M_i$ for some $i<\mu$. Since $M_i \preccurlyeq^u M_{i+1}$, $p$ is realised in $M_{i+1}$, hence in $M$.
\end{proof}

We now recall some concepts and facts about first order stable theories relevant to our main arguments.

\begin{definition}[{\cite[III.3.1]{sh:c}}]\label{fo-kappa-def}
    Given a first order complete theory $T$, $\kappa(T)$ is the least $\kappa$ such that for all $\subseteq$-increasing sequences of sets $\langle A_i : i \leq \kappa\rangle$, and all $p \in \gS(\bigcup_{i<\kappa} A_i)$, there exists $i < \kappa$ such that $p$ does not fork over $A_i$. If no such $\kappa$ exists, $\kappa(T) = \infty$.

    $\kappa_r(T)$ is the first regular $\kappa \geq \kappa(T)$ which is regular (so $\kappa(T)$ or $\kappa(T)^+$), or $\infty$ if $\kappa(T) = \infty$.
\end{definition}

Shelah's original definition was slightly different, but they are equivalent by a short argument using \cite[III.3.2]{sh:c} and monotonicity of non-forking. Observe that for $\mu$ regular, $\mu < \kappa(T)$ if and only if $\mu < \kappa_r(T)$.

\begin{definition}[{\cite[Definition VII 3.2(A)]{sh:c}}]\label{uniform-trees}
    We say $\{ \varphi_\alpha : \alpha < \beta, \alpha$ is a successor ordinal $\}$ with $\{\ba_\eta : \eta \in {}^{\leq \beta}\theta\}$ is a \emph{uniform $\beta$-tree} provided that 
    \begin{enumerate}
        \item $\varphi_\alpha$ is a formula for all successor $\alpha < \beta$, and $\theta$ is an infinite ordinal
        \item $\mon \models \varphi_\alpha(\ba_\eta, \ba_{\eta\upharpoonright \alpha})$ for all successor $\alpha < \beta$ and all $\eta \in {}^{\beta}\theta$
        \item For all $\alpha$ such that $\alpha + 1 < \beta$, $\eta \in {}^\alpha \theta$, and $\nu \in {}^\beta \theta$, $\ba_\nu$ realises at most $1$ of the formulas $\varphi_{\alpha + 1}(\bx, \ba_{\eta^{\wedge} i})$ where $i < \theta$
    \end{enumerate}
    
\end{definition}

\begin{remark}
	Some notes on the definition:
	\begin{enumerate}
		\item If a uniform $\beta$-tree exists for one (infinite) $\theta$, by compactness there is a uniform $\beta$-tree for any other $\theta$ with the same $\varphi_\alpha$'s. 
		\item For zero and limit $\alpha<\beta$, it doesn't matter what $\ba_\nu$ is for all $\nu \in {}^\alpha \theta$.
		\item As a result of conditions (2) and (3), for $\eta \in 2^\beta$ and $\alpha+1 < \beta$, we have $\mon \models \varphi_{\alpha+1}(\ba_\eta, \ba_{\eta \upharpoonright \alpha ^\wedge i})$ if and only if $\eta(\alpha) = i$. In fact, this is the only role of conditions (2) and (3) in our argument for Theorem \ref{fo-short-models-are-distinct}; see equation ($\dagger$) in the proof.
		\item Shelah's original definition is slightly broader: condition (3) is weakened to say that for all successor $\alpha < \beta$, there exist $m_\alpha \in \omega$ such that at most $m_\alpha$ of the formulas are realised. Since our results only require a tree with $m_\alpha = 1$ for all appropriate $\alpha$, which we can guarantee exists by {\cite[VII 3.7(2)]{sh:c}} cited below, we restrict to this case for simplicity.
	\end{enumerate}
\end{remark}

When superstability fails (that is, $\kappa(T) > \aleph_0$), the nicest possible forms of uniform $\beta$-trees exist for $\beta < \kappa(T)$:

\begin{fact}[{\cite[VII 3.7(2)]{sh:c}}]\label{factsh}
    Suppose $T$ is stable. If $\beta < \kappa(T)$, then $T$ has a uniform $\beta$-tree.
\end{fact}

We will also use that high cofinality unions of saturated models retain saturation. This was originally noted by Harnik as a by-product his proof of \cite[2.1]{harnik} (pointed out to him by Shelah); a more explicit proof can be found in {\cite[III.3.11]{sh:c}}.

\begin{fact}\label{fo-long-saturated-lims-are-saturated}
    Suppose $T$ is stable. If $\langle M_i : i < \delta \rangle$ is an elementary chain of $\lambda$-saturated models where $\cof(\delta) \geq \kappa(T)$, then $\bigcup_{i < \delta} M_i$ is $\lambda$-saturated.
\end{fact}

\section{Main results}\label{maintheoremsection}

First, we show that `high limits are isomorphic'.

\begin{lemma}\label{fo-long-lims-are-iso}
    Let $T$ be a complete first order theory stable in $\lambda \geq |T|$. If $\kappa_r(T) \leq \mu < \lambda^+$ and $\mu$ is regular, then any $(\lambda, \mu)$-limit model is saturated. Hence if $\kappa_r(T) \leq \mu_1 \leq \mu_2 < \lambda^+$, where $\mu_1$ and $\mu_2$ are both regular, then the $(\lambda, \mu_1)$-limit model and the $(\lambda, \mu_2)$-limit model are isomorphic.
\end{lemma}

\begin{proof}
    Let $M$ be a $(\lambda, \mu)$-limit model. Build a continuous sequence $\langle M_i : i \leq \mu \rangle$ such that for all $i < \mu$, $M_{i+1}$ is saturated and universal over $M_i$. This is possible because under stability, given $M_i$ we can find a universal extension $M_i'$ of size $\lambda$ by Fact \ref{limits-exist}, and then find saturated $M_{i+1}$ extending $M_i'$ of size $\lambda$ by \cite[III.3.12]{sh:c} (note we still have $M_i \preccurlyeq^u M_{i+1}$). As $M_\mu$ is a $(\lambda, \mu)$-limit model, we can assume $M = M_\mu$ by Fact \ref{cfiso}. Then $M = \bigcup_{i < \mu} M_{i+1}$, and since $\mu \geq \kappa_r(T) \geq \kappa(T)$, $M$ is saturated by Fact \ref{fo-long-saturated-lims-are-saturated}.

    The latter part of the statement follows from \cite[I.1.11]{sh:c}, which says that saturated models of the same cardinality are isomorphic.
\end{proof}

Now we prove that `short limits are non-isomorphic'. The argument is inspired by \cite[6.1]{gvv}. 

\begin{theorem}\label{fo-short-models-are-distinct}
    Let $T$ be a complete first order theory stable in $\lambda \geq |T|$. Let $\mu_1, \mu_2$ be regular infinite cardinals where $\mu_1 < \kappa_r(T)$ and $\mu_1 < \mu_2 < \lambda^+$. Then the $(\lambda, \mu_1)$-limit model and $(\lambda, \mu_2)$-limit model are not isomorphic. 
\end{theorem}

\begin{proof}[Proof of Theorem \ref{fo-short-models-are-distinct}]
    Let $\mon$ is the monster model of $T$. We will show that the $(\lambda, \mu_1)$-limit is not $\mu_2$-saturated. Since the $(\lambda, \mu_2)$-limit is $\mu_2$-saturated, this is enough. 
    
    Let $\theta = (2^\lambda)^+$. Since $\mu_1 < \kappa_r(T)$ is regular, $\mu_1 < \kappa(T)$, so by Fact \ref{factsh} there exists a uniform $\mu_1$-tree comprised of $\{ \varphi_\alpha : \alpha < \mu_1, \alpha \text{ successor}\}$ and $\{\ba_\eta : \eta \in {}^{\leq \mu_1}\theta\}$. Conditions (2) and (3) of Definition \ref{uniform-trees} imply that for all $\eta \in {}^{\mu_1}\theta, \alpha < \mu_1, i < \theta$, 
    \begin{equation}\label{tree_eq}\tag{$\dagger$}
    \mon \models \varphi_{\alpha+1}(\ba_\eta, \ba_{(\eta \upharpoonright \alpha)^{\wedge} i}) \iff \eta(\alpha) = i
    \end{equation}

    Now we construct an increasing continuous sequence of models $\langle M_\alpha : \alpha < \mu_1 \rangle$, an increasing continuous sequence of maps $\langle \eta_\alpha : \alpha < \mu_1 \rangle$ in ${}^{< \mu_1}\theta$, and another (not necessarily increasing or continuous) sequence of maps $\langle \nu_\alpha ; \alpha < \mu_1 \rangle$ in ${}^{\leq \mu_1}\theta$ such that for all $\alpha < \mu_1$
    \begin{enumerate}
        \item $M_{\alpha + 1}$ is universal over $M_\alpha$ and $\|M_\alpha\| = \lambda$
        \item $\eta_\alpha, \nu_\alpha \in {}^\alpha \theta$, and $\ba_{\eta_\alpha}, \ba_{\nu_ \alpha} \in M_{\alpha + 1}$
        \item $\eta_{\alpha+1}\upharpoonright \alpha = \nu_{\alpha+1}\upharpoonright \alpha = \eta_\alpha$, and $\eta_{\alpha+1}(\alpha) \neq \nu_{\alpha+1}(\alpha)$
        \item $\operatorname{tp}({\ba_{\eta_\alpha}}/M_\alpha) = \operatorname{tp}({\ba_{\nu_\alpha}}/M_\alpha)$ when $\alpha$ is a successor.
    \end{enumerate}

    \vspace{5px}

    For some intuition, think of $\eta = \eta_{\mu_1}$ as the main `branch' up the tree, and the $\nu_\alpha$ as `offshoots' or `thorns' of the main branch corresponding to the tuples $\ba_{\nu_\alpha}$ that have the same type as the main branch $\ba_{\eta_\alpha}$  over $M_\alpha$. Note $\ba_{\eta}$ believes the formulas $\varphi_\alpha$ along the main branch, but on none of the thorns, by (\ref{tree_eq}). The type that encapsulates this `branch with thorns', the type of $\ba_\eta$ over all $\ba_{\eta_\alpha}$ and $\ba_{\nu_\alpha}$ for $\alpha < \mu_1$, will be the one we show is omitted in $M_{\mu_1}$.
    
    \textbf{This is possible:} For $\alpha = 0$, let $M_0$ be any model of size $\lambda$, and let $\eta_0 = \nu_0 = \langle \rangle$.
    
    For limit $\alpha$, take $M_\alpha = \bigcup_{\beta < \alpha} M_\beta$, $\eta_\alpha = \bigcup_{\beta < \alpha} \eta_\beta$, and $\nu_\alpha \in {}^\alpha \theta$ (the choice of $\nu_\alpha$ doesn't matter).

    For successor $\alpha = \beta+1$, suppose we have defined $M_\beta, \eta_\beta, \nu_\beta$. Take $M_\alpha$ any model universal over $M_\beta$ containing $\ba_{\eta_\beta}$ and $\ba_{\nu_\beta}$. Now $|S(M_\beta)| \leq 2^\lambda < \theta$ (in fact we could have used $\theta = \lambda^+$ by stability in $\lambda$). So there exist $i \neq j < \theta$ such that $\operatorname{tp}(\ba_{\eta_{\beta}{}^{\wedge}i}/M_\beta) = \operatorname{tp}(\ba_{\eta_{\beta}{}^{\wedge}j}/M_\beta)$. Let $\eta_\alpha = \eta_{\beta}{}^{\wedge}i$ and $\nu_\alpha = \eta_{\beta}{}^{\wedge}j$. This completes the construction.

    \vspace{10pt}

    \textbf{This is enough:} Now, let $M = \bigcup_{\alpha < \mu_1} M_\alpha$ and $\eta = \bigcup_{\alpha < \mu_1} \eta_\alpha$. Note that $M$ is a $(\lambda, \mu_1)$-limit model by condition (1) of the construction. Let $p = \{\varphi_{\alpha+1}(\bx, \ba_{\eta_{\alpha+1}}) \land \neg \varphi_{\alpha+1}(\bx, \ba_{\nu_{\alpha+1}}) : \alpha < \mu_1 \}$. By Equation (\ref{tree_eq}), $\ba_\eta \models p$, so $p$ is consistent. This type's domain is of cardinality at most $\mu_1 < \mu_2$, so it is enough to show $M$ fails to realise $p$.
    
    Suppose for contradiction that $p$ is realised in $M$, say $\bb \in M$ and $\bb \models p$. Then there is some $\alpha < \mu_1$ such that $\bb \in M_{\alpha+1}$. As $\bb \models p$, $\mon \models \varphi_{\alpha+1}(\bb, \ba_{\eta_{\alpha+1}})$ and $\mon \models \neg \varphi_{\alpha+1}(\bb, \ba_{\nu_{\alpha+1}})$. As $\bb \in M_{\alpha + 1}$, this implies $\operatorname{tp}(\ba_{\eta_{\alpha+1}}/M_{\alpha+1}) \neq \operatorname{tp}(\ba_{\nu_{\alpha+1}}/M_{\alpha+1})$, which contradicts condition (4) of our construction.
    
    Hence $p$ is not realised in $M$. So the $(\lambda, \mu_1)$-limit model is not $\mu_2$-saturated, as claimed, and therefore is not isomorphic to the $(\lambda, \mu_2)$-limit model by Fact \ref{long-limits-are-saturated}.
\end{proof}

The following is our full characterisation of the spectrum of limit models for a stable theory. 

\begin{corollary}\label{fo-main-result}
    Let $T$ be a $\lambda$-stable complete theory where $\lambda \geq |T|$. Suppose $T$ is $\lambda$-stable where $\lambda \geq |T|$. Suppose $\delta_1, \delta_2 < \lambda^+$ are limit ordinals with $\cof(\delta_1) < \cof(\delta_2)$. Then for any $N_1, N_2\models T$ where $N_l$ is a $(\lambda, \delta_l)$-limit model for $l = 1, 2$, 
    \[N_1 \text{ is isomorphic to } N_2 \iff \cof(\delta_1) \geq \kappa_r(T)\]
\end{corollary}

\begin{proof}
    By Fact \ref{cfiso}, we may assume $\delta_1, \delta_2$ are regular. Then the result is then immediate from Lemma \ref{fo-long-lims-are-iso} and Theorem \ref{fo-short-models-are-distinct}.
\end{proof}

\begin{corollary}
    Let $T$ be a $\lambda$-stable complete theory where $\lambda \geq |T|$. If $\kappa_r(T) = \aleph_\alpha \leq \lambda$, there are exactly $|\alpha|+1$ limit models.\footnote{Here $|\alpha| + 1$ should be interpretted as cardinal addition.}
\end{corollary}

\begin{proof}
    Note that there are $|\alpha| + 1$ regular cardinals $\leq \aleph_\alpha$ for every ordinal $\alpha$. The theorem shows that the limit models of these lengths are all distinct. On the other hand, all $(\lambda, \mu)$-limit models for regular $\mu \in [\aleph_\alpha, \lambda^+)$ are isomorphic to the $(\lambda, \aleph_\alpha)$-limit model by Fact \ref{fo-long-lims-are-iso}.
\end{proof}

\subsection*{Remarks}

We end with some remarks concerning the main result and it's applications.

\begin{remark}\label{weaker-description}
    We describe how our results relate to those in \cite{bema}.
    \begin{enumerate}
        \item Lemma \ref{fo-long-lims-are-iso} is a weakening of \cite[3.1]{bema}. It restricts us to the first order setting, and we cannot guarantee that the isomorphism fixes a base model, but we avoid developing the technical machinery of towers from \cite[\textsection 3]{bema}, which reduces the proof from 13 pages to a third of a page. 
        \item Similarly, Theorem \ref{fo-short-models-are-distinct} is a weakening of \cite[4.1]{bema}. It restricts us to the first order setting, but avoids the technical complication of relating (first order) non-forking to $\lambda$-non-splitting.
        \item Corollary \ref{maintheorem} is a slight weakening of \cite[6.11(4)]{bema} (we have not shown an isomorphism fixing a base model exists if the limits share the same base). \cite[6.11(4)]{bema} is itself a corollary of \cite[5.1]{bema}, which gives a characterisation of the spectrum of limit models for `nice' stable AECs with a strong enough independence notion. The advantage of our proof of Corollary \ref{maintheorem} is that we replaced the complex machinery of \cite[\textsection 3, \textsection 4, \textsection 5]{bema}, as well as the work showing that first order stable theories satisfy the conditions of \cite[5.1]{bema} in \cite[\textsection 6]{bema},  with much simpler arguments using familiar tools of first order stable theories.
    \end{enumerate}

\end{remark}

\begin{remark}\label{ramiremark}
     The method of \cite[6.1]{gvv} proved that the $(\lambda, \omega)$-limit model isn't $\aleph_1$-saturated, whereas the $(\lambda, \mu)$-limits are $\lambda$-saturated for $\mu \geq \kappa(T)$, to establish there were non-isomorphic limit models. The proof of Theorem \ref{fo-short-models-are-distinct} is similar, but we replace $\omega$ and $\lambda$ by arbitrary regular $\mu_1 < \mu_2 \leq \kappa(T)$. The trade-off is this forces us to use Fact \ref{factsh} as stated (\cite{gvv} uses a stronger version of uniform $\omega$-tree that is easily constructed from the normal kind, but this is not possible as far as we can tell for $\beta > \omega$). However, the weaker form of uniform $\beta$-tree we use is enough.
\end{remark}

\begin{remark}
    Shelah constructs uniform $\beta$-trees in \cite[VII 3.7(2)]{sh:c} (i.e. the formulas $\varphi_\alpha$ and constants $\ba_\eta$) from first order dividing. It is not clear to the author that this argument can be translated into an AEC setting using Galois types rather than syntactic types.
\end{remark}

\begin{remark}
    Corollary \ref{fo-main-result}, together with Theorem \cite[5.1]{bema} applied with $\dnf$ as first order non-forking, give an alternative proof of Lemma \cite[6.22(2)]{bema}, which essentially says that the two notions of the local character cardinal agree. Borrowing the notation of \cite[2.36]{bema} for the universal local character cardinal of an independence relation, note that $\kappa_r(T) \leq \kappa(\dnf, \operatorname{Mod}(T)_\lambda, \preccurlyeq^u)$ if and only if the $(\lambda, \kappa_r(T))$-limit model and $(\lambda,\kappa(\dnf, \operatorname{Mod}(T)_\lambda, \preccurlyeq^u))$-limit model are isomorphic by Corollary \ref{fo-main-result}. By the same argument but using \cite[5.1]{bema} in place of Corollary \ref{fo-main-result}, this holds if and only if $\kappa_r(T) \geq \kappa(\dnf, \operatorname{Mod}(T)_\lambda, \preccurlyeq^u)$. Since either $\kappa_r(T) \leq \kappa(\dnf, \operatorname{Mod}(T)_\lambda, \preccurlyeq^u)$ or $\kappa_r(T) \geq \kappa(\dnf, \operatorname{Mod}(T)_\lambda, \preccurlyeq^u)$, both must be true. Therefore $\kappa_r(T) = \kappa(\dnf, \operatorname{Mod}(T)_\lambda, \preccurlyeq^u)$.
\end{remark}

\printbibliography

\end{document}